\documentclass{article}

\usepackage{a4wide}
\usepackage{amsfonts}
\usepackage{amsmath}
\usepackage{amsthm}
\usepackage[english]{babel}
\usepackage{bbm}

\newcommand{\C}{\mathbb C}
\newcommand{\R}{\mathbb R}
\renewcommand{\Re}{\operatorname{Re}}
\renewcommand{\Im}{\operatorname{Im}}
\newcommand{\diag}{\operatorname{diag}}
\renewcommand{\i}{\mathrm{i}}
\newcommand{\tr}{\operatorname{tr}}

\newtheorem{theorem}{Theorem}[section]
\newtheorem{lemma}[theorem]{Lemma}
\newtheorem{proposition}[theorem]{Proposition}
\newtheorem{observation}[theorem]{Observation}
\newtheorem{corollary}[theorem]{Corollary}

\theoremstyle{definition}
\newtheorem{definition}[theorem]{Definition}
\newtheorem{problem}[theorem]{Problem}

\theoremstyle{remark}
\newtheorem{remark}[theorem]{Remark}
\newtheorem{example}[theorem]{Example}

\newtheorem{conjecture}[theorem]{Conjecture}

\begin{document}

\title{A singular M-matrix perturbed by a nonnegative rank one matrix has positive principal minors; is it D-stable?}
\author{Joris Bierkens\footnotemark[1] \ and Andr\'e Ran\footnotemark[2]}
\renewcommand{\thefootnote}{\fnsymbol{footnote}}
\footnotetext[1]{Donders Institute for Brain, Cognition and Behaviour, Radboud University Nijmegen, Heyendaalseweg 135,
6525 AJ Nijmegen,
The Netherlands.
E-mail: \texttt{j.bierkens@science.ru.nl}.
The research leading to these results has received funding from the European Community's Seventh Framework Programme (FP7/2007-2013) under grant agreement no. 270327.
}
\footnotetext[2]{Department of Mathematics, FEW, VU university Amsterdam, De Boelelaan
    1081a, 1081 HV Amsterdam, The Netherlands
    and Unit for BMI, North-West~University,
Potchefstroom,
South Africa. E-mail:
    \texttt{a.c.m.ran@vu.nl}}
\date{}

\maketitle

\begin{abstract}
The positive stability and D-stability of singular M-matrices, perturbed by (non-trivial) nonnegative rank one perturbations, is investigated. In special cases positive stability or D-stability can be established. In full generality this is not the case, as illustrated by a counterexample. However, matrices of the mentioned form are shown to be P-matrices.
\end{abstract}



\section{Introduction}

\subsection{Problem formulation}
Let $H \in \C^{n \times n}$.
Let $\rho(H)$ denote the spectral radius of $H$, i.e. \[\rho(H) = \max\{ |\lambda| :\lambda \in \C, \det (\lambda I - H) = 0 \}.\] We consider the matrix $A := \rho(H) I - H$. Let $v, w \in \C^n$. In this paper we deal with the stability properties of the matrix $B = \rho(H) I - H + v w^{\star} = A + v w^{\star}$. In particular, we may ask the following questions:

\begin{problem}
\label{prob:general}
\begin{itemize}
\item[(i)] Under what conditions is $B$ strictly positive stable?
\item[(ii)] Let $C \in \C^{n\times n}$. Under what conditions is $C B$ strictly positive stable?
\end{itemize}
\end{problem}

Our primary interest is in the case where $H$ is real nonnegative, in which case $A$ is a singular M-matrix, and the perturbation $v w^T$ is positive, i.e. a matrix having only positive entries.

We encountered this problem in the process of studying stability of a particular type of ordinary differential equation. This motivation is discussed briefly in Section~\ref{sec:origin}, and with some more attention to detail in Appendix~\ref{app:origin_of_problem}. 
  We will provide a few general observations in Section~\ref{sec:general_observations}. These observations are not  restricted to the class of nonnegative matrices, and will help us to eliminate some trivial cases. Still by means of introduction, in Section~\ref{sec:normal} we will pay attention to the elementary case where $H$ is not assumed to be nonnegative but instead symmetric (or more generally, normal), and the rank one perturbation symmetric. 

The nonnegative case is more challenging. Our main question is whether a matrix of the mentioned form is D-stable, i.e. stable even when multiplied by any positive diagonal matrix. This problem turns out to be difficult, and we can not yet conclusively answer this question. However we are able to show stability or even D-stability in certain special cases, which may be considered the main results of this paper (Section~\ref{sec:nonnegative}). A counterexample is provided to show that we cannot hope to establish D-stability for the general class of singular M-matrices with positive rank one perturbations. To provide direction for further research, we state as a conjecture that if $A$ is a  singular and symmetric M-matrix, which geometrically simple eigenvalue 0, and if $v w^T$ is a positive rank one perturbation, then $A + v w^T$ is D-stable; see Section~\ref{sec:conjecture} for the counterexample and the conjecture.
A result that is potentially useful in this direction, and interesting in its own right, is that all matrices in this class are P-matrices, i.e. matrices with positive principal minors (Section~\ref{sec:pmatrices}).

The problem is part of a wider context: in the last two decades the study of eigenvalues and Jordan structure of rank-one perturbations of matrices
has seen rapid development. For unstructured matrices we refer to \cite{DM,HM,RW,Sav1,Sav2}. For matrices that exhibit structure in the setting of an
indefinite inner product space several surprising results have been obtained, see \cite{FGJR,MMRR1,MMRR2,MMRR3}. From this point of view the problem we
study here is natural: given an M-matrix $A=\rho(H)I-H$, and given positive vectors $v$, $w$, it is natural to ask whether or not all eigenvalues of 
$A+vw^T$ are in the open right half plane. 

\subsection{Notation and definitions}
The euclidean norm in $\C^n$ is denoted by $|\cdot|$, with corresponding inner product $\langle \cdot, \cdot \rangle$. The matrix norm induced by $|\cdot|$ is denoted by $||\cdot||$. The identity matrix is denoted by $I$. If $x_1, \hdots, x_m$ are vectors in $\C^n$, then $[x_1, \hdots, x_m]$ denotes the matrix in $\C^{n\times m}$ that has $x_1, \hdots, x_m$ as its columns, in the specified order. Entries of a matrix $A \in \C^{n \times n}$ or a vector $x \in \C^n$ are denoted by $(a_{ij})_{i,j = 1, \hdots, n}$ or $(x_i)_{i = 1, \hdots, n}$, respectively.
For a matrix $A$ we denote its \emph{spectrum} by $\sigma(A) = \left\{ \lambda \in \C : \det(\lambda I - A) = 0 \right\}$ and its \emph{spectral radius} by $\rho(A) = \max \left\{ |\lambda| : \lambda \in \sigma(A) \right\}$.
We say that a matrix $A$ is \emph{(strictly) positive stable} if $\Re \lambda \geq 0$ ($\Re \lambda > 0$) for all $\lambda \in \sigma(A)$.
If a matrix $A$ or a vector $x$ has only nonnegative (positive) entries, we write $A \geq 0, (A > 0), x \geq 0, (x > 0)$, respectively.
A \emph{positive diagonal matrix} is a diagonal matrix with positive entries on the diagonal.

%
%

\subsection{Origin of the problem: a dynamical system for the solution of an eigenvalue problem}
\label{sec:origin}
Suppose $H, C \in \R^{n \times n}$ with $C$ nonsingular. For $z \in \R^n$, $z \neq 0$, let $P_z = \frac{z z^T}{|z|^2}$. Note that for $z \in \R^n$, $P_z$ is the matrix corresponding to orthogonal projection onto the span of $z$. Consider the following coupled system of ordinary differential equations.
\begin{equation} \label{eq:ode} \left\{ \begin{array}{ll} \dot z(t) & = \left(I - P_{z(t)} \right) C (H - \lambda(t) I) z(t), \\
                          \dot \lambda(t) & = z(t)^T C \left(H - \lambda(t) I \right)z(t),
                         \end{array} \right.
                         \quad (t \geq 0.)
\end{equation}

It turns out that the equilibrium points of~\eqref{eq:ode} are exactly the pairs $(z, \lambda)$ with $z$ an eigenvector corresponding to the eigenvalue $\lambda$ of $H$. Furthermore $(z, \lambda)$ is a locally (asymptotically) stable equilibrium point if and only if $C \left( \lambda I - H + z z^T \right)$ is (strictly) positive stable. The proofs of these results are relegated to Appendix~\ref{app:origin_of_problem}. We are especially interested in the case where $\lambda = \rho(H) \in \sigma(H)$, which is the case for $H$ nonnegative and for $H$ symmetric positive semidefinite.

\subsection{General observations}
\label{sec:general_observations}
%

\begin{observation} \label{obs:geometrically_simple} Suppose $\rho(H)$ is an eigenvalue of $H \in \C^{n \times n}$, but is not geometrically simple. Then $\rho(H) I - H + v w^{\star}$ is not strictly positive stable for any $v, w \in \C^n$.
\end{observation}

\begin{proof}
In this case $\dim \left(\ker(\rho(H) I - H) \right)\geq 2$. Since a perturbation of rank one can only increase the rank of a matrix by one, $\rho(H) I - H + v w^{\star}$ is not strictly positive stable. 
\end{proof}

In view of this observation, we will always assume that the zero eigenvalue of $A$ is geometrically simple.
First we briefly consider the special case in which $v$ is a right eigenvector of $H$ corresponding to the eigenvalue $\rho(H)$.

\begin{proposition}
\label{prop:spectrum_v_eigenvector}
Let $H \in \C^{n \times n}$, and suppose that $\rho(H)$ is a eigenvalue of $H$ with right eigenvector $v$. Let $B = \rho(H) I - H + v w^{\star}$, with $w \in \C^n$.
Then 
\begin{itemize}
 \item If $\rho(H)$ is algebraically simple, then \[ \sigma(B) = \sigma(\rho(H) I - H) \backslash \{0\} \cup \{ w^{\star} v \}.\]
In particular if $\Re w^{\star} v \geq 0$ (resp. $\Re w^{\star} v > 0$) then $B$ is positive stable (resp. strictly positive stable).
\item If $\rho(H)$ is not algebraically simple, then
\[ \sigma(B) = \sigma(\rho(H) I - H) \cup \{ w^{\star} v \}.\]
In particular $0 \in \sigma(B)$, so that $B$ is not strictly positive stable. If $\Re w^{\star} v \geq 0$ then $B$ is positive stable.
\end{itemize}

\end{proposition}
By taking Hermitian adjoints, we could equivalently consider the case where $w$ is a left eigenvector of $H$. This result is not new (see e.g. \cite{Ding2007} and likely many other references) but should be mentioned in the context of this paper. 
\begin{proof}
Assume $\rho(H)$ is an algebraically simple eigenvalue of $H$ with right eigenvector $v \in \R^n$. We may write the Jordan decomposition of $H$ as $H = T J T^{-1}$ where the first column of $T$ is equal to $v$. Then 
\[ T^{-1} B T = \rho(H) I - J + \begin{bmatrix} 1 & 0 & \hdots & 0 \end{bmatrix}^T (w^{\star} T).\]
Note that the last term only contributes to the first row, so that $T^{-1} B T$ is an upper triangular matrix. Therefore the eigenvalues of $A$ are characterized by the diagonal elements of $T^{-1} B T$, with as topleft element $(w^{\star} T)_1 = w^{\star} v$, and the other elements equal to the non-zero eigenvalues of $\rho(H) I - J$, which all have positive real parts.

In case $\rho(H)$ is an eigenvalue of $H$ that is not algebraically simple, the argument is analogous. Only in this case $\rho(H)I - J$ has more than one diagonal element equal to zero, so that $0 \in \sigma(B)$.
\end{proof}

The following lemma will be useful in the remainder of this paper.
\begin{lemma}
\label{lem:nonsingular}
Suppose zero is a geometrically simple eigenvalue of $A$, with left- and right eigenvectors $z_{\ell}^{\star} \neq 0$ and $z_{r}\neq 0$ (i.e. $z_{\ell}^{\star} A = 0$ and $A z_r = 0$). Then $A + v w^{\star}$ is nonsingular if and only if $z_{\ell}^{\star} v \neq 0$ and $w^{\star} z_{r} \neq 0$.
\end{lemma}

\begin{proof}
Assume first that $z_{\ell}^{\star} v \neq 0$ and $w^{\star} z_{r} \neq 0$ and suppose that $x \in \ker (A + v w^{\star})$. In particular $z_{\ell}^{\star} v w^{\star} x = z_{\ell}^{\star} (A + v w^{\star}) x = 0$. This implies that $w^{\star} x = 0$ (since $z_{\ell}^{\star} v \neq 0$). However then $A x = (A + v w^{\star}) x = 0$, so that $x = \gamma z_{r}$ for some $\gamma \in \C$, using that 0 is a geometrically simple eigenvalue. Then $w^{\star} x = \gamma w^{\star} z_{r} = 0$, which implies that $\gamma = 0$, or equivalently that $x = 0$. 

Conversely, if $w^{\star} z_{r} = 0$, then $(A + v w^{\star}) z_{r} = 0$; similarly $z_{\ell}^{\star}(A + vw^{\star}) = 0$ if $z_{\ell}^{\star} v = 0$.
\end{proof}

This lemma shows that the following (`non-zero projection') condition is necessary for strict positive stability.
\begin{equation}
\label{eq:nonzero_projection} \tag{NZP} (z_{\ell}^{\star} v) (w^{\star} z_r) \neq 0.
\end{equation}

\subsection{Normal matrices}
\label{sec:normal}
For $H \in \C^{n \times n}$ normal (in the sense that $H H^{\star} = H^{\star} H$), and the rank one perturbation Hermitian (i.e. equal to $v v^{\star}$ for some $v \in \C^n$), the situation is quite straightforward. We will only require one lemma.

\begin{lemma}
\label{lem:spectrum_HplusHtranspose}
Let $H \in \C^{n \times n}$ be a normal matrix. Then 
\begin{itemize}
 \item[(i)] $\rho(H + H^{\star}) = 2 \max \left\{ |\Re \mu | : \mu \in \sigma(H) \right\} \leq 2 \rho(H)$. 
 \item[(ii)] $2 \rho(H) \in \sigma(H + H^{\star})$ if and only if $\rho(H) \in \sigma(H)$,
 \item[(iii)] for $x \in \C^n$, we have that $(H + H^\star) x = 2 \rho(H) x$ if and only if $H x = \rho(H) x$.
\end{itemize}
\end{lemma}

\begin{proof}
Since $H$ is normal, it is unitarily equivalent to a diagonal matrix, say $H = U \Lambda U^{\star}$. Therefore \[ U^{\star} (H + H^{\star}) U = U^{\star} H U + \left( U H U^{\star}\right)^{\star} = \Lambda + \Lambda^{\star}.\] 
This shows that $\sigma (H + H^{\star}) = \left\{ 2 \Re \mu : \mu \in \sigma(H) \right\}$, and in particular (i) follows. Also an eigenvector of $H + H^{\star}$ corresponding to the eigenvalue $2 \Re \mu$ of $H+H^{\star}$, for $\mu \in \sigma(H)$, is equal to an eigenvector of $H$ corresponding to the eigenvalue $\mu$ of $H$. 
If $\rho(H) \in \sigma(H)$, then $2 \rho(H) \in \sigma(H+H^{\star})$ and if $H x = \rho(H) x$, then $(H+H^{\star})x = 2 \rho(H) x$.
Now let $x \in \C^n$, such that $(H+H^{\star})x = 2 \rho(H) x$. If $x = 0$, trivially $H x = \rho(H) x$. Suppose $x \neq 0$. For all eigenvalues $2 \Re \mu$ of $H + H^{\star}$ such that $\mu \neq \rho(H)$, we necessarily have that $\Re \mu < \rho(H)$, so that $2 \Re \mu < 2 \rho(H)$. Therefore, if $(H+H^{\star}) x = 2 \rho(H) x$, then $x$ is an eigenvector corresponding to the eigenvalue $\rho(H)$ of $H$.
\end{proof}

\begin{proposition}
\label{prop:normal}
Suppose that $H \in \C^{n \times n}$ is normal and $v \in \C^n$. Let $B = \rho(H) I - H + v v^{\star}$ and let $C \in \C^{n \times n}$ be positive definite. In the following cases, $CB$ is strictly positive stable.
\begin{itemize}
 \item[(i)] $\rho(H)$ is a geometrically simple eigenvalue of $H$ with eigenvector $z$ and~\eqref{eq:nonzero_projection} holds (i.e. $v^{\star} z \neq 0$);
 \item[(ii)] $\rho(H) \notin \sigma(H)$.
\end{itemize}
\end{proposition}

\begin{proof}
We write $\lambda := \rho(H)$.
\begin{itemize}
 \item[(i)] By Lemma~\ref{lem:spectrum_HplusHtranspose} (i), (ii), we have
\[ \rho(H+H^{\star}) = 2 \rho(H) = 2 \lambda \in \sigma(H+H^{\star}).\] In particular $2 \lambda I - H - H^{\star}$ is positive semidefinite.
Let $x \in \C^n$. Then
\begin{align*}
\left \langle x, \left( 2 \lambda I - H - H^{\star} + 2 v v^{\star}\right) x \right \rangle \geq 0.
\end{align*}
Suppose the inequality holds with equality. Then both $(H + H^{\star})x= 2 \lambda x$ and $\langle v, x \rangle = 0$.
The first equality implies, by Lemma~\ref{lem:spectrum_HplusHtranspose} (iii), that $H x = \lambda x$. Since $\lambda$ is a geometrically simple eigenvalue, $x = \mu z$ for some $\mu \in \C$. Hence $\mu v^{\star} z = v^{\star} x = 0$ gives that $\mu = 0$, so $x = 0$.
We may conclude that $B + B^{\star} = 2 \lambda I - H - H^{\star} + 2 v v^{\star}$ is positive definite.
For $C$ positive definite, we have that
$ C^{-1} C B +B^{\star} C C^{-1} = B + B^{\star}$ is positive definite. By the Lyapunov theorem \cite[Theorem 2.2.1]{HornJohnson1994}, $C B$ is strictly positive stable.
 \item[(ii)] By Lemma~\ref{lem:spectrum_HplusHtranspose} (i), $\rho(H + H^{\star}) \leq 2 \rho(H)$, so that $2 \lambda I -H - H^{\star}$ is positive definite. In particular, $2 \lambda I - H - H^{\star} + 2 v v^{\star}$ is positive definite, and we may follow the proof of (i) to deduce that $C B$ is strictly positive stable.
\end{itemize}
\end{proof}

\section{Nonnegative case}
\label{sec:nonnegative}
In this section we consider the case where $H \geq 0$ and $v, w \in \R^n$. The main results are Theorem~\ref{thm:main}, in which positive stability is established in certain special cases, and Corollary~\ref{cor:d-stable}, which states the implications of Theorem~\ref{thm:main} regarding D-stability (as defined below) under certain conditions.

Inevitably, this section will strongly depend on the theory of M-matrices; see \cite[Chapter 6]{BermanPlemmons1994} or \cite[Section 2.5]{HornJohnson1994}.

\begin{definition} An \emph{M-matrix} is a matrix of the form $\gamma I - P$, where $P \in \R^{n \times n}$, $P \geq 0$, and $\gamma \geq \rho(P)$.
We call $A \in \R^{n\times n}$ \emph{D-stable} if $D A$ is strictly positive stable for any positive diagonal matrix $D$.
\end{definition}
M-matrices $\gamma I - P$, $P \geq 0$, where $\gamma > \rho(P)$ will be explicitly referred to as \emph{nonsingular M-matrices} here, in contrast to \emph{singular M-matrices} for which $\gamma = \rho(P)$.
Throughout this section, let $A = \rho(H) I - H \in \R^{n \times n}$ denote a singular M-matrix, with $H \geq 0$, and let $v$, $w \in \R^n$.

We will focus our attention to the following problem:

\begin{problem}
\label{prob:D-stable}
Under what conditions is $A + v w^T$ D-stable?
\end{problem}

\begin{remark}
We are considering stability under left-multiplication by a diagonal matrices, since both M-matrices and rank two matrices behave well under multiplication with a positive diagonal matrix. For multiplication on the left by more general matrices, we cannot expect to achieve general results.
\end{remark}

The solution of Problem~\ref{prob:D-stable} is our ultimate goal for this section. Unfortunately we are not yet going to succeed in establishing general conditions for $A + v w^T$ to be D-stable. A weaker version of Problem~\ref{prob:D-stable} is needed: instead of investigating positive stability of $D(A+vw^T)$ for all positive diagonal matrices $D$, we fix $D$.

\begin{problem}
\label{prob:stable_for_given_D}
Suppose $D$ is a positive diagonal matrix. Under what conditions is $D(A + v w^T)$ strictly positive stable?
\end{problem}

We may write $D(A + v w^T) = \widetilde A + \widetilde v w^T$, where $\widetilde A = D A$ and $\widetilde v = D v$. Then $\widetilde A$ is a singular M-matrix. We may therefore consider the following simpler but equivalent problem.

\begin{problem}
\label{prob:stable}
Under what conditions is $A + v w^T$ strictly positive stable?
\end{problem}

Recall that we assume that the zero eigenvalue of $A$ is geometrically simple (see Observation~\ref{obs:geometrically_simple}), with $z_{\ell} \geq 0$ and $z_{r} \geq 0$ left- and right eigenvectors of $A$ corresponding to the eigenvalue 0, respectively (the nonnegativity of $z_{\ell}$, $z_{r}$ being a consequence of Perron's theorem, \cite[Theorem 8.3.1]{HornJohnson1990}). If $A$ is symmetric (or more generally, normal), we will write $z = z_{\ell} = z_{r}$. As usual, $A = \rho(H) I - H$ for $H \geq 0$.

\begin{remark}
A sufficient condition for the zero eigenvalue of $A$ to be algebraically simple, and hence geometrically simple, is that $A$ is irreducible (by the Perron-Frobenius theorem \cite[Theorem 8.4.4]{HornJohnson1990}). If this is the case, $z_{\ell}$ and $z_{r}$ have strictly positive entries. It is not a necessary condition, see e.g. Example~\ref{ex:two_dimensional} (b) below. The irreducible case is the typical case from the perspective of applications. 
\end{remark}

In partial solution to Problem~\ref{prob:stable}, we establish the following result, which is the main result of this section.

\begin{theorem}
\label{thm:main}
Let $A=\rho(H)I-H \in \R^{n \times n}$ be a singular M-matrix. Suppose that 0 is a geometrically simple eigenvalue of $A$. Let $v, w \in \R^n$ and suppose~\eqref{eq:nonzero_projection} holds. Then $A + v w^T$ is strictly positive stable in any of the following cases:
\begin{itemize}
 \item[(i)] $n = 2$, $v, w \geq 0$, and one of the following conditions hold:
 \begin{itemize}
 \item[(a)] the zero eigenvalue of $A$ is \emph{algebraically} simple;
\item[(b)] $w^T v > 0$.
\end{itemize}
\item[(ii)] $A v = 0$, or $w^T A = 0$;
\item[(iii)] $A$ is normal (in particular, if $A$ is symmetric) and $w = v$;
\item[(iv)] $v, w > 0$ and $2 h_{ij} \geq v_i w_j$ for all $i, j$ (where $h_{ij}$ denotes the $i,j$-entry of $H$, and likewise $v_i$, respectively, $w_j$, denote the $i^{th}$ entry of $v$, respectively, the $j^{th}$ entry of $w$);
\item[(v)] There exists a nonnegative matrix $K \in \R^{n \times n}$, such that $H - v w^T \leq K$ (where the inequality is meant to hold element-wise), and
\begin{equation} \label{eq:Fan_inequality} \rho(K) + h_{ii} < \rho(H) + k_{ii} + v_i w_i \quad \mbox{for all $i = 1, \dots, n$.}\end{equation}
\item[(vi)] $A$ is symmetric, and $\frac{||v|| \ ||w||}{2} < \left(\frac{\alpha}{1-\alpha} \right)^{1/2} \tau$, where $\tau = \min \left\{ \mu \in \sigma(A) : \mu \neq 0 \right\}$ and 
\begin{equation} 
\label{eq:alpha} \alpha :=\frac{ \left| \langle v, z  \rangle \langle w, z \rangle \right|}{||v|| \ ||w|| \ ||z||^2},
\end{equation}
where $z=z_{\ell}=z_{r}$. We assume here that $v$ and $w$ are not both parallel to $z$, so that $\alpha < 1$. (If either $v \parallel z$ or $w \parallel z$, then (ii) of this theorem may be applied to deduce strict positive stability).
\end{itemize}
\end{theorem}

\begin{proof}
\begin{itemize}
\item[(i)] For $n=2$, if 0 is an algebraically simple eigenvalue of $A$, it follows that $A$ has one other eigenvalue $\mu \neq 0$. Then $\mu$ is real-valued because complex eigenvalues of real matrices would come in conjugate pairs, and since $A$ is an M-matrix, $\mu > 0$. In either of the cases (algebraically simple eigenvalue, or $w^T v > 0$) we have $\tr(A + v w^T) = \tr(A)+\tr(vw^T)=
 \mu + w^T v > 0$. Furthermore
\begin{align*}
\det (A + v w^T) & = (a_{11} + v_1 w_1) (a_{22} + v_2 w_2) - (a_{12} + v_1 w_2) (a_{21} + v_2 w_1) \\
 & =v_1 w_1 a_{22} + v_2 w_2 a_{11} -a_{12} v_2 w_1 - a_{21} v_1 w_2 \geq 0
\end{align*}
using that $\det(A) = 0$, $v \geq 0$, $w \geq 0$.
By Lemma~\ref{lem:nonsingular} and the conditions $z_{\ell}^T v \neq 0$, $w^T z_{r} \neq 0$, we have in fact $\det(A + v w^T) > 0$.

It follows that $\tr(A + v w^T) > 0$ and $\det(A + v w^T) > 0$, so that the eigenvalues of $A + v w^T$ are in the open right half plane.
\item[(ii)] This is a special case of Proposition~\ref{prop:spectrum_v_eigenvector}.
\item[(iii)] This is a special case of Proposition~\ref{prop:normal}.
\item[(iv)] The comparison matrix $M[B] = [m_{ij}]$ of a matrix $B$ is defined as (see \cite{HornJohnson1994}, Definition 2.5.10)
\[ m_{ij} = \left\{ \begin{array}{ll} |b_{ij}| \quad & \mbox{if} \ i = j, \\ - |b_{ij}| \quad & \mbox{if} \ i \neq j. \end{array} \right. \]
Applied to our case, the comparison matrix $M = M [A + v w^T]$ has entries
\[ m_{ij} = \left\{ \begin{array}{ll} |a_{ii} + v_i w_i| = a_{ii} + v_i w_i, \quad & \mbox{for} \ i = j, \ \mbox{and} \\
                     - |a_{ij} + v_i w_j| = - |v_i w_j - h_{ij}| \geq - h_{ij} = a_{ij}, \quad & \mbox{for} \ i \neq j,
                    \end{array} \right. \]
where we used $2h_{ij} \geq v_i w_j$.
Note that, since $2 h_{ij} \geq v_i w_j > 0$ for all $i,j$, by Perron's theorem $z_{r} > 0$.
Pick $E = \diag(z_{r})$, then $M E$ is strictly row diagonally dominant (using $v_i w_i > 0$), and so $M$ is an M-matrix (\cite[Theorem 2.5.3]{HornJohnson1994}). Furthermore, $A + v w^T$ has positive entries on the diagonal: $a_{ii} + v_i w_i > a_{ii} \geq 0$, using that $A$ has nonnegative entries on the diagonal and $v, w$ are positive. Therefore $A + vw^T$ is an H-matrix with positive entries on the diagonal, so that it is strictly positive stable (see \cite{HornJohnson1994}, p. 124).
\item[(v)] The proof consists of an application of Fan's theorem, \cite[Theorem 8.2.12]{HornJohnson1990}: for a nonnegative matrix $K$ and a matrix $M$ satisfying $|m_{ij}| \leq k_{ij}$ for all $i,j =1, \dots, n$, the eigenvalues of $M$ are contained in 
\[ \bigcup_{i=1}^n \{ z \in \C : |z - m_{ii}| \leq \rho(K) - k_{ii} \}.\]
Applying this result for $M = H - v w^T$, we find that for any eigenvalue $\lambda \in \C$ of $M$,
\[ \Re \lambda \leq \max_{i} \left[ m_{ii} + \rho(K) - k_{ii} \right] = \max_{i} \left[h_{ii} - v_i w_i + \rho(K) - k_{ii} \right] < \rho(H).\]
Hence the spectrum of $\rho(H) I - H + v w^T = \rho(H) I - M$ is contained in the open right halfplane.
\item[(vi)] For the proof of this result, a number of lemmas are needed. The proof will be continued below.
\end{itemize}
\end{proof}

\begin{remark}
Theorem~\ref{thm:main} (iv) may also be considered as a consequence of Theorem~\ref{thm:main} (v), by taking $K = H$. However, the proof of (iv) given above will play a role in the proof of Corollary~\ref{cor:d-stable}.
\end{remark}

\begin{lemma}
\label{lem:spectrumMmatrixplusrankone}
Let $A \in \R^{n \times n}$ and $v, w \in \R^n$.
Then for all $\mu \in \C$, $\mu \notin \sigma(A)$, we have that $\mu \in \sigma(A + v w^T)$ if and only if $w^T (\mu I -A)^{-1} v = 1$.
\end{lemma}

\begin{proof}
Let $\mu \in \C$, $\mu \notin \sigma(A)$, so that $\mu I - A$ is invertible.
It may be verified that, if $w^T (\mu I - A)^{-1} v \neq 1$, the inverse of $\mu I - (A + v w^T)$ exists and is given by
\[ (\mu I - A)^{-1}\left(I + \frac {v w^T} {1 - w^T (\mu I - A)^{-1} v}  (\mu I - A)^{-1} \right),\]
which may be verified by computation; see also \cite[Section 0.7.4]{HornJohnson1990}.
So in this case $\mu \notin \sigma(A + v w^T)$.
Conversely, if $w^T (\mu I - A)^{-1} v = 1$ then $v \neq 0$ and $(\mu I - A - v w^T)(\mu I - A)^{-1} v = 0$, so that $\mu \in \sigma (A + v w^T)$.
\end{proof}

\begin{lemma}
\label{lem:real_eigenvalues_nonnegative}
Suppose $A$ is a (singular or nonsingular) M-matrix, $v, w \in \R^n$, $v \geq 0$, $w \geq 0$. Suppose $\mu$ is a real-valued eigenvalue of $A + v w^T$. Then $\mu \geq 0$.
\end{lemma}

\begin{proof}
Suppose $\mu < 0$. Note that $\mu \notin \sigma(A)$. Now $A - \mu I$ is a nonsingular M-matrix, so $\left(A - \mu I \right)^{-1} \geq 0$, as can be seen by expanding the inverse as a Neumann series; see also \cite[Theorem 2.5.3]{HornJohnson1994}. If $\mu \in \sigma(A  + v w^T)$, then by Lemma~\ref{lem:spectrumMmatrixplusrankone}, $w^T(\mu I -  A)^{-1} v = 1$, which is a contradiction.
\end{proof}

\begin{lemma}
\label{lem:notimmediately}
Let $A$ be a singular M-matrix and suppose that $0$ is an algebraically simple eigenvalue of $A$.
Let $v$, $w \in \R^n$, $v, w \geq 0$, and suppose that condition~\ref{eq:nonzero_projection} holds.
Define a matrix-valued curve $\Gamma(t): t \mapsto A + t v w^T$, $t \in \R$.
There exists a $t_0 > 0$ such that $\Gamma(t)$ is strictly positive stable for $t \in (0, t_0]$.
\end{lemma}

\begin{proof}
Since $A$ is an M-matrix, the real parts of the nonzero eigenvalues are positive. Let $\varepsilon > 0$ such that zero is the only eigenvalue of $\Gamma(0)$ in a open disc $B(0, \varepsilon) \subset \C$ of radius $\varepsilon$ around 0 (and note that the eigenvalue has algebraic multiplicity one). By continuous dependence of the eigenvalues on the entries of a matrix \cite[Appendix D]{HornJohnson1990}, we may choose $t_1$ sufficiently small so that the number of eigenvalues of $\Gamma(t)$ in $B(0, \varepsilon)$ remains equal to one for $0 \leq t < t_1$. Let $0 \leq t < t_1$ and write $\lambda(t)$ for the unique element in $\sigma(\Gamma(t)) \cap B(0, \varepsilon)$. Since $\Gamma(t)$ is realvalued, it follows that $\lambda(t) \in \R$ (otherwise there would be a conjugate eigenvalue, also in $B(0, \varepsilon)$). By Lemma~\ref{lem:real_eigenvalues_nonnegative}, $\lambda(t) \geq 0$. By condition~\eqref{eq:nonzero_projection} and Lemma~\ref{lem:nonsingular} in fact $\lambda(t) > 0$ for all $0 \leq t < t_1$. Now choose $t_0 \leq t_1$ so that all other eigenvalues of $\Gamma(t)$ remain to be contained in the open right half plane for $0 \leq t \leq t_0$ (again using the continuity of the eigenvalues as function of $t$).
\end{proof}

\begin{lemma}
\label{lem:estimate_imaginary_eigenvalue}
Let $A \in \R^{n \times n}$ be positive semidefinite. Suppose that zero is an algebraically simple eigenvalue of $A$ and let $\tau$ the smallest nonzero eigenvalue of $A$.
Let $z$ be an eigenvector of $A$ corresponding to the zero eigenvalue. Let $v, w \in \R^n$, $v \neq 0$, $w \neq 0$.
Suppose $\mu = b \i$, with $b > 0$, is an eigenvalue of $A + v w^T$. Let $\alpha$ be given by~\eqref{eq:alpha}. Then
\[ \left( \frac{ \alpha}{1 - \alpha}\right)^{1/2} \tau \leq b \leq \frac{ ||v|| \ ||w||}{2}.\]
\end{lemma}

\begin{proof}
Let $0 = \lambda_1 < \tau = \lambda_2 \leq \lambda_3 \leq \hdots \lambda_n$ denote the eigenvalues of $A$, listed multiple times according to their algebraic multiplicity. The eigenvalues of $(A - b \i I)^{-1}$ are
\[ (\lambda_k - b \i)^{-1} = \frac{\lambda_k  + b \i}{\lambda_k^2 + b^2}, \quad k = 1, \hdots, n,\]
so that the eigenvalues of $\Im (A - b \i I)^{-1}$ are given by
\begin{equation}
\label{eq:imaginary_parts} \frac{b}{\lambda_k^2 + b^2}, \quad k = 1, \hdots, n, 
\end{equation}
and the eigenvalues of $\Re (A - b \i I)^{-1}$ are given by
\begin{equation}
 \label{eq:real_parts} \frac{\lambda_k}{\lambda_k^2 + b^2}, \quad k = 1, \hdots, n.
\end{equation}
\emph{Proof of the lower bound for $b$:}
Let $P = z z^T / ||z||^2$ denote the orthogonal projection in the direction of $z$. Note that $A$ and $P$ commute (in fact $AP = PA = 0$), and 
\[ P (b \i I-A)^{-1} P = -\frac {\i} b P .\]
Furthermore,
\begin{equation} \label{eq:estimate_imaginary_part_residual} \left|\left| \Im (I-P) (b \i I-A)^{-1} (I-P) \right|\right| = \frac{b}{\tau^2 + b^2},\end{equation}
using that the norm of a symmetric real matrix is equal to its largest eigenvalue. We have $P (b \i I - A)^{-1} (I-P) =(b \i I - A)^{-1} P(I-P) = 0$, and analogously $(I-P) (b \i I - A)^{-1} P = 0$, so that  
\begin{align} \nonumber 0 & = \Im w^T (b \i I-A)^{-1} v \\
\nonumber & =  \Im \left[ w^T P(b \i I-A)^{-1} Pv +w^T(I-P)(b \i I-A)^{-1} Pv + w^T P(b \i I-A)^{-1} (I-P)v  \right. \\
\nonumber & \quad \quad \quad \left. +  w^T (I-P) (b \i I-A)^{-1} (I-P) v \right] \\
\nonumber & = \Im \left[ w^T P (b \i I-A)^{-1} P v +  w^T (I-P) (b \i I-A)^{-1} (I-P) v \right] \\
\label{eq:required_equality} & = -\frac{\langle w, z \rangle \langle v, z \rangle}{b ||z||^2} + \Im w^T (I-P) (b \i I-A)^{-1} (I-P) v.
\end{align}
In order for $\Im w^T (b \i  I-A)^{-1} v = 0$ (as required by Lemma~\ref{lem:spectrumMmatrixplusrankone}), we therefore require that
\[ \frac{\langle w, z \rangle \langle v, z \rangle}{b ||z||^2} = \Im w^T (I-P) (b \i I-A)^{-1} (I-P) v.\]
The righthand side may, by~\eqref{eq:estimate_imaginary_part_residual}, be estimated as 
\[ \left| \Im w^T (I-P) (b \i I-A)^{-1} (I-P) v \right| \leq ||v|| \ ||w|| \frac{b}{\tau^2 + b^2}.\]
Hence
\[  \frac{ \left| \langle w, z \rangle \langle v, z \rangle \right| }{b ||z||^2} \leq ||v|| \ ||w|| \frac{b}{\tau^2 + b^2},\] which is seen to be equivalent to the lower bound by elementary algebraic manipulation.

\emph{Proof of the upper bound for $b$:}
By~\eqref{eq:real_parts},
\[ ||\Re (b \i I-A)^{-1}|| \leq \max_{\xi \geq 0} \frac{\xi}{\xi^2 + b^2} = \frac 1 {2 b}.\]
Then the upper bound follows since we require, by Lemma~\ref{lem:spectrumMmatrixplusrankone}, that
\[ 1 = \Re w^T (b \i I-A)^{-1} v \leq \frac{||v||  \ ||w||}{2 b}.\]
\end{proof}

\begin{proof}[Proof of Theorem~\ref{thm:main} (vi):]
Consider the mapping $t \mapsto \Gamma(t) := A + t v w^T$ for $0 \leq t \leq 1$. It is our aim to show that $\Gamma(1)$ is strictly positive stable. 
Suppose now that $\Gamma(1)$ is not strictly positive stable, so that some eigenvalue of $\Gamma(1)$ lies in the closed left halfplane.
By Lemma~\ref{lem:notimmediately}, there exists a $t_0 > 0$ so that $\Gamma(t)$ is strictly positive stable for $0 < t \leq t_0$. If $t_0 \geq 1$ then we reach a contradiction with the assumption that $\Gamma(1)$ is not stable, so we may further assume that $t_0 < 1$. By continuous dependence of the eigenvalues of $\Gamma(t)$ on $t$ (see \cite[Appendix D]{HornJohnson1990}, the trajectory of at least one eigenvalue must cross the vertical axis in the complex plane, i.e. there exists a time $t_1 \geq 0$,  $t_0 < t_1 \leq 1$ such that $\Gamma(t_1)$ has a purely imaginary eigenvalue $\mu = \i b$ for some $b \in \R$. Because $\Gamma(t_1)$ is nonsingular by Lemma~\ref{lem:nonsingular} $b \neq 0$. Since $\Gamma(t_1)$ is real, there will exist a pair of conjugate eigenvalues, so we may assume $b > 0$. Note that $\alpha$ as given by~\eqref{eq:alpha} does not depend on the scaling of $v w^T$ by a scalar, so that the value of $\alpha$ for $\Gamma(t_1)$ is the same as that for $\Gamma(1)$. Also note that $\alpha < 1$, unless both $v$ and $w$ are multiples of $z$. By Lemma~\ref{lem:estimate_imaginary_eigenvalue}, $b \geq \tau \sqrt{\frac{\alpha}{1 - \alpha}}$. At the same time, $b \leq \frac{||v|| \ ||w|}{2}$. Under the assumptions of Theorem 2.7 (vi), this is a contradiction, so that $\Gamma(1)$ is strictly positive stable.
\end{proof}

\begin{example}
\label{ex:two_dimensional}
We consider some examples illustrating the conditions of case (i) of Theorem~\ref{thm:main}. 
\begin{itemize}
 \item[(a)] Let $A = \begin{bmatrix} 0 & -1 \\ 0 & 0 \end{bmatrix}$. Then $A$ is a singular M-matrix and 0 is geometrically simple but not algebraically simple. It has left- and right eigenvectors $z_{\ell}^T = \begin{bmatrix} 0 & 1 \end{bmatrix}$ and $z_{r} = \begin{bmatrix} 1 & 0 \end{bmatrix}^T$. Let $v = z_{\ell}$ and $w = z_{r}$. Then the conditions~\eqref{eq:nonzero_projection}, $n = 2$ and $v, w \geq 0$ of (i) are satisfied but $A$ is not algebraically simple and $v^T w = 0$. We have $A + v w^T = \begin{bmatrix} 0 & -1 \\ 1 & 0 \end{bmatrix}$ which has eigenvalues $\pm \i$; in particular, it is not strictly positive stable.
 \item[(b)] Suppose $A = \begin{bmatrix} 1 & -1 \\ 0 & 0 \end{bmatrix}$. Then $A$ is a singular M-matrix with eigenvalues 0 and 1. It has left-and right zero eigenvectors $z_{\ell}^T = \begin{bmatrix} 
 0 & 1 \end{bmatrix}$ and $z_{r} = \begin{bmatrix} 1 & 1 \end{bmatrix}^T$. Let $v = \begin{bmatrix} 1 & 0 \end{bmatrix}^T$ and $w = z_{r}$. Then all conditions of (i) are satisfied, except $z_{\ell}^T v > 0$, i.e.~\eqref{eq:nonzero_projection} does not hold. We have $A + v w^T = \begin{bmatrix} 2 & 0 \\ 0 & 0 \end{bmatrix}$ which has rank 1; in particular it is not strictly positive stable. This example also illustrates that the condition $z_{\ell}^T v\neq 0$, $w^T z_r \neq 0$ cannot be replaced by the condition $v^T w \neq 0$.
\end{itemize}
\end{example}

Theorem~\ref{thm:main} implies the following result concerning D-stability of $\rho(H) I - H + v w^T$.
\begin{corollary}
\label{cor:d-stable}
Let $A = \rho(H) I - H \in \R^{n \times n}$ be a singular M-matrix. Suppose that 0 is a geometrically simple eigenvalue of $A$. Let $v, w \in \R^n$ and suppose~\eqref{eq:nonzero_projection} holds. Then $A + v w^T$ is D-stable in any of the following cases:
\begin{itemize}
 \item[(i)] $n = 2$, $v, w \geq 0$, and one of the following conditions hold:
 \begin{itemize}
 \item[(a)] $A$ is irreducible;
 \item[(b)] $v, w > 0$.
 \end{itemize}
\item[(ii)] $v, w > 0$ and one of the following equivalent conditions holds:
\begin{itemize}
\item[(a)] There exists a positive diagonal matrix $D$ such that $D A D^{-1}$ and $D v w^T D^{-1}$ are symmetric;
\item[(b)] There exists a positive diagonal matrix $E$ such that $EA$ and $E v w^T$ are symmetric;
\item[(c)] For $E = \diag(w) \diag(v)^{-1}$, $EA$ is symmetric.
\end{itemize}
\item[(iii)] $v, w > 0$ and $2 h_{ij} \geq v_i w_j$ for all $i, j$.
\end{itemize}
\end{corollary}

\begin{proof}
\begin{itemize}
 \item[(i)] Note that (i-a) and (i-b) are invariant under left-multiplication by a positive diagonal matrix, and imply conditions (i-a) and (i-b) of Theorem~\ref{thm:main}, respectively.
\item[(ii)] To see the equivalence of the stated conditions under (ii), note that (ii-c) trivially implies (ii-b). We will show (ii-b) $\Rightarrow$ (ii-a) and (ii-a) $\Rightarrow$ (ii-c). Assume (ii-b) holds. Then $EA$ is congruent to $E^{1/2} A E^{-1/2}$, which is therefore symmetric; similarly for $E v w^T$. Now assume (ii-a) holds. Then $D^2 v w^T = D D v w^T D^{-1} D$ is symmetric. In particular, for all $i, j$, $d^2_i v_i w_j = d_j^2 v_j w_i$. Manipulation gives $d_i^2 v_i / w_i = d_j^2 v_j/w_j$ for all $i, j$, so that $d_i^2 v_i/w_i$ is constant in $i$. This shows that $D = c \diag(w)^{1/2} \diag(v)^{-1/2} = c E^{1/2}$ for some $c > 0$, with $E$ as in (ii-c). Now $E^{1/2} A E^{-1/2} = D A D^{-1}$ is symmetric, so by congruence, $E A$ is symmetric, and similarly for $E v w^T$, so that (ii-c) holds. Now (ii-a) states that $A + v w^T$ is similar to a matrix satisfying the conditions of Theorem~\ref{thm:main} (iii), so that it has the same spectrum.
\item[(iii)] By the proof of Theorem~\ref{thm:main} (iv), $A + v w^T$ is an H-matrix. Then it is D-stable, see \cite[Page 124]{HornJohnson1994}.
\end{itemize}

\end{proof}

\subsection{Counterexample and conjectures}
\label{sec:conjecture}
In view of Theorem~\ref{thm:main}, one might wonder whether \emph{any} $\rho(H) I - H + v w^T$, with $H$ nonnegative with $\rho(H)$ as geometrically simple eigenvalue, $v > 0$ and $w > 0$, is strictly positive stable, or equivalently D-stable. 
It is a challenging task to come up with numerical counterexamples to such a conjecture by randomly generated matrices. However, counterexamples do exist, as illustrated by the following example. 

\begin{example}
Let $H$, $v$ and $w$ be given by
\begin{equation}
\label{eq:counterexample}
H = \begin{bmatrix} 1.0 & 1.0 & 1.0 & 1.0 & 1.0 & 1.0 \\ 0 & 0.5 & 1.0 & 1.0 & 1.0 & 1.0 \\ 0 & 0 & 0 & 0.35 & 1.0 & 1.0 \\ 0 & 0 & 0.35 & 0 & 1.0 & 1.0 \\ 0 & 0 & 0 & 0 & 0 & 0.15 \\ 0 & 0 & 0 & 0 & 0.15 & 0 \end{bmatrix}, \quad v = \begin{bmatrix} 0.05 \\ 0.10 \\ 0.05 \\ 0.15 \\ 0.25 \\ 0.20 \end{bmatrix}, \quad w = \begin{bmatrix} 0.65 \\ 0.15 \\ 0.10 \\ 0.05 \\ 0.30 \\ 0.80 \end{bmatrix}.
\end{equation}
The spectral radius of $H$ is $\rho(H) = 1.000$ (up to four decimal places) and the zero eigenvalue of $\rho(H) I - H$ is algebraically simple. Condition~\eqref{eq:nonzero_projection} is satisfied since $v, w > 0$. However, the spectrum of $\rho(H) I - H + v w^T$ is given, in up to four decimals, by $-0.0093 \pm 0.9949 \i, 1.1649 \pm 0.2223 \i, 1.3460, 1.1377$. In particular $\rho(H) I - H + v w^T$ is \emph{not} strictly positive stable.

Note that $H$ is not irreducible; however it can be made irreducible by setting $h_{61} = \varepsilon > 0$. For $\varepsilon$ sufficiently small, e.g. $\varepsilon = 10^{-6}$, this does not alter the given approximations of the eigenvalues of $H$. Therefore strengthening the condition on $H$, demanding irreducibility, will not remedy this counterexample.
\end{example}

In view of this counterexample, we have to lower our expectations. We raise the following conjecture, which adds the requirement that $H$ is symmetric.

\begin{conjecture}
\label{conj:symmetric}
Suppose $H \in \R^{n\times n}$, $H \geq 0$ is symmetric, with geometrically simple eigenvalue $\rho(H)$, $v, w \in \R^n$, $v > 0$, $w > 0$. Then $\rho(H) I - H + v w^T$ is D-stable.
\end{conjecture}

As a potentially significant step towards a better understanding of this problem, we also raise the following conjecture (which might apply to the general case, i.e. also in case $A$ is not symmetric).
\begin{conjecture}
Let $\Gamma(t)$ be as defined in Lemma~\ref{lem:notimmediately}. There exists a $t_1>0$ such that for $t>t_1$ the matrix $\Gamma(t)$ is strictly positive stable.
\end{conjecture}

\subsection{Positive principal minors}
\label{sec:pmatrices}

The following result could play a part in a proof of Conjecture~\ref{conj:symmetric}. As is well known, D-stability is implied by Lyapunov diagonal stability \cite[Theorem 2.5.8]{HornJohnson1994}, where $A$ is said to be \emph{Lyapunov diagonally stable} if there exists a positive diagonal matrix $D$ such that $A^T D + D A$ is positive definite. Matrices that are Lyapunov diagonally stable are examples of P-matrices, i.e. they have positive principal minors \cite[Theorem 6.2.3]{BermanPlemmons1994}. We show below that the class of matrices we are concerned with, consists of P-matrices.

\begin{lemma}
\label{lem:p0matrix}
Let $A$ be a singular M-matrix and let $v, w \in \R^n$,  $v \geq 0$, $w \geq 0$.
Then the matrix $A + v w^T$ is a $\textrm{P}_0$-matrix, i.e. all its principal minors are nonnegative.
\end{lemma}

\begin{proof}
Write $Q = v w^T$.
By \cite{BermanPlemmons1994}, Theorem 6.4.6 ($A_3$), it suffices to show that $A + Q + D$ is nonsingular for all positive diagonal matrices $D$. So let $D$ be a positive diagonal matrix and suppose $(A + Q + D)x = 0$ for some $x \in \R^n$. Write $B = A + D$. Then $B$ is a nonsingular M-matrix (by \cite{BermanPlemmons1994}, Theorem 6.2.3 ($C_{11}$)). Since $(A+Q+D)x = (B + v w^T)x = 0$, we have that $x = \eta B^{-1} v$ for $\eta = -w^T x$. If $v = 0$, then $x = 0$. Suppose $v \neq 0$. Then
\[ 0 = (B + Q)x = \eta (B + Q) B^{-1}v = \eta v + \eta Q B^{-1} v.\]
Since $Q \geq 0$ and $B^{-1} \geq 0$ (because it is an inverse M-matrix), $Q B^{-1} \geq 0$. Together with $v \geq 0$, $v \neq 0$, and the above equality, this implies that $\eta = 0$, i.e. $x = \eta B^{-1} v = 0$.
\end{proof}


\begin{proposition}
\label{prop:pmatrix}
Let $A \in \R^{n \times n}$ be a singular, irreducible M-matrix and let $v, w \in \R^n$, $v \geq 0$, $w \geq 0$, satisfying~\eqref{eq:nonzero_projection}.
Then $A + v w^T$ is a P-matrix, i.e. all principal minors are positive.
\end{proposition}

\begin{proof}
Write $Q = v w^T$.

Let $\beta \subset \{1, \hdots, n\}$ be a proper subset of $A$ and consider the principal submatrix $A[\beta] + Q[\beta]$.
We have 
\[ \det(A[\beta] + Q[\beta]) = \det(A[\beta]) \det ( I + A[\beta]^{-1} Q[\beta]),\]
and we will show that the two factors are positive.
By \cite{BermanPlemmons1994}, Theorem 6.4.16, $A[\beta]$ is a nonsingular M-matrix, using irreducibility of $A$. 
Note that $A[\beta]^{-1} Q[\beta]$ is by the assumption on $Q$ a matrix of rank one. Since $A[\beta]$ is an M-matrix, $A[\beta]^{-1}$ and hence $A[\beta]^{-1} Q[\beta]$ are nonnegative matrices and therefore we can write $A[\beta]^{-1} Q[\beta] = \gamma x y^T$, with $\gamma \geq 0$, $x, y \geq 0$, and $\langle x, y \rangle = 1$.
The eigenvalues of $I + A[\beta]^{-1} Q[\beta]$ are then $1$ (with eigenspace $y^{\perp}$) and $1 + \gamma$ (with eigenvector $x$), and hence $\det(I + A[\beta]^{-1} Q[\beta]) > 0$. Since $A[\beta]$ is a nonsingular M-matrix, also $\det(A[\beta])>0$. So all proper principal minors are positive and it remains to check that $\det(A + Q) > 0$. By Lemma~\ref{lem:nonsingular}, $A +Q$ is nonsingular. Since, by Lemma~\ref{lem:p0matrix}, $A + Q$ is a $P_0$-matrix, it follows that $\det(A+Q) > 0$.
\end{proof}

\section{Conclusion}
Under certain conditions stability or D-stability can be established for singular M-matrices, perturbed non-trivially by a rank one matrix. As illustrated by a counterexample, this is not possible in general. However we have shown that these matrices are `close' to being positive stable, in the sense that all principal minors are positive.

\subsection*{Acknowledgement}
We thank the anonymous referee for his detailed inspection of our manuscript, resulting in various comments and corrections that have helped to eliminate errors and improve the readibility of this work.

\appendix
\section{Details on the origin of the problem}
\label{app:origin_of_problem}
Suppose $H, C \in \R^{n \times n}$ with $C$ nonsingular. For $z \in \R^n$, $z \neq 0$, let $P_z = \frac{z z^T}{|z|^2}$. Note that for $z \in \R^n$, $P_z$ is the matrix corresponding to orthogonal projection onto the span of $z$. Consider the following coupled system of ordinary differential equations.
\begin{equation} \tag{\ref{eq:ode}} \left\{ \begin{array}{ll} \dot z(t) & = \left(I - P_{z(t)} \right) C (H - \lambda(t) I) z(t), \\
                          \dot \lambda(t) & = z(t)^T C \left(H - \lambda(t) I \right)z(t),
                         \end{array} \right.
                         \quad (t \geq 0.)
\end{equation}

%
\begin{observation}
Suppose $(z(t),\lambda(t))$, $t \in \R$, satisfies~\eqref{eq:ode}. Then $\frac{d}{dt} |z(t)|^2 = 0$. In particular, if $|z(t_0)| = 1$ for some $t_0 \in \R$, then $|z(t)| = 1$ for all $t \in \R$.
\end{observation}

Starting from $z(0) \in \R^n$ with $|z(0)| = 1$, we may therefore consider~\eqref{eq:ode} as defining a flow on $S^{n-1} \times \R$ (with $S^{n-1}$ the unit sphere in $\R^n$). The equilibrium points of~\eqref{eq:ode} may be characterized as follows.
\begin{observation}
\label{obs:equilibrium}
$(z, \lambda)$ is an equilibrium point of~\eqref{eq:ode} if and only if $z$ is an eigenvector of $H$ with eigenvalue $\lambda$.
\end{observation}

\begin{proof}
The point $(z,\lambda) \in S^{n-1} \times \R$ is an equilibrium point of~\eqref{eq:ode} if and only if
\[ (I - P_z) C (H - \lambda I ) z = 0, \quad \mbox{and} \quad z^T C(H - \lambda I) z = 0.\]
This is the case if and only if $C(H - \lambda I) z = \left( (I - P_z) + z z^T \right) C (H -\lambda I) z = 0$. Since $C$ is nonsingular, the assertion follows.
\end{proof}

According to Observation~\ref{obs:equilibrium}, the system~\eqref{eq:ode} has the potential to locate eigenvalues and eigenvectors of $H$. It is then natural to investigate the stability of the equilibrium points.

The following proposition shows that the stability of an equilibrium point depends on the stability properties of a matrix of the type considered in this paper.

\begin{proposition}
\label{prop:stable_equilibrium}
Let $(z, \lambda)$ denote an equilibrium point of~\eqref{eq:ode}, with $|z| = 1$. Then this equilibrium point is locally (asymptotically) stable if and only if $C \left( \lambda I - H + z z^T \right)$ is (strictly) positive stable.
\end{proposition}

\begin{proof}
%
Suppose $(z, \lambda)$ is an equilibrium point of~\eqref{eq:ode}. 
By straightforward computation, the linearization of~\eqref{eq:ode} in $(z, \lambda)$ is given by the matrix
\begin{equation} \label{eq:linearization} \begin{bmatrix} \left(I - P_{z} \right) C \left( H - \lambda I \right) & - \left(I - P_{z} \right) C z \\
    z^T C \left(H - \lambda I \right) & - z^T C z
   \end{bmatrix},\end{equation}
where we used that $(H - \lambda I)z = 0$.
This linearization is with respect to the standard basis in $\R^n \times \R$. However, we are interested in the flow on $S^{n-1} \times \R$.
Choose $v_1,\hdots, v_{n-1}$ in $\R^n$ such that $\{v_1, \hdots, v_{n-1}, z\}$ is an orthonormal system in $\R^n$. The set $\{ v_1, \hdots, v_{n-1}\}$ may be identified in a canonical way with an orthonormal basis of the tangent space of $S^{n-1}$ in $z$, denoted by $T_{z} S^{n-1}$. Leaving the basis in the one-dimensional factor $\R$ (where $\lambda$ resides) unchanged, we have thus obtained a basis of the tangent space of $S^{n-1} \times \R$ in the point $(z, \lambda)$, consisting of $n$ vectors
\[ w_1 = \begin{bmatrix} v_1 \\ 0 \end{bmatrix}, \dots, w_{n-1} = \begin{bmatrix} v_{n-1} \\ 0, \end{bmatrix}, w_n = \begin{bmatrix} 0 \\ 1 \end{bmatrix}.\]
Let $V = \begin{bmatrix} v_1 & \dots & v_{n-1} \end{bmatrix}$, the matrix with $v_1, \hdots, v_{n-1}$ as columns. A vector in the canonical basis in $\R^n \times \R$ may be expressed in the new basis $w_1, \dots, w_n$ by multiplying it from the left by 
\[ \begin{bmatrix} V^T & 0 \\ 0 & 1 \end{bmatrix} \in \R^{n \times (n+1)}.\]
The inverse operation (i.e. transforming a vector that is expressed with respect to the basis $w_1, \dots, w_n$, into its representation with respect to the canonical basis of $\R^n \times \R$) consists of left-multiplication by
\[ \begin{bmatrix} V & 0 \\ 0 & 1 \end{bmatrix} \in \R^{(n+1) \times n}.\]
Note that $V V^T = I - z z^T$. Now the linearization~\eqref{eq:linearization} of~\eqref{eq:ode} can be expressed with respect to the basis $w_1, \dots, w_n$, by multiplying on the left and right by 
\[ \begin{bmatrix} V^T & 0 \\ 0 & 1 \end{bmatrix} \quad \mbox{and} \quad \begin{bmatrix} V & 0 \\0 & 1 \end{bmatrix},\]
respectively. Using that $V^T P_z = 0$ by the definition of $V$, this gives
\[
\begin{bmatrix} V^T & 0 \\ 0 & 1 \end{bmatrix} \begin{bmatrix} \left(I - P_{z} \right) C \left( H - \lambda I \right) & - \left(I - P_{z} \right) C z \\
    z^T C \left(H - \lambda I \right) & - z^T C z \end{bmatrix}
   \begin{bmatrix} V & 0 \\ 0 & 1 \end{bmatrix} = \begin{bmatrix} V^T C(H - \lambda I) V & -V^T C z \\  z^T C (H - \lambda I)V  & - z^T C z \end{bmatrix}.\]
The equilibrium point $(z, \lambda)$ is locally (asymptotically) stable if and only if the negative of this matrix is (strictly) positive stable. By performing an orthonormal basis transformation by $\begin{bmatrix} v_1 & \dots & v_{n-1} & z \end{bmatrix} = \begin{bmatrix} V & z \end{bmatrix}$, we find that this is equivalent to the requirement that the matrix
\begin{align*}
- \begin{bmatrix} V & z \end{bmatrix} \begin{bmatrix} V^T C(H - \lambda I) V & -V^T C z \\  z^T C (H - \lambda I)V  & - z^T C z \end{bmatrix} \begin{bmatrix} V^T \\ z^T \end{bmatrix} & = - \begin{bmatrix} C(H-\lambda I) V & - C z \end{bmatrix} \begin{bmatrix} V^T \\ z^T \end{bmatrix} \\
& = C \left( \lambda I - H + z z^T \right).
\end{align*}
is (strictly) positive stable.
\end{proof}

\begin{remark}
The matrix $C$ may be interpreted as an unknown perturbation of the identity matrix. It may be the case that some equilibrium point $(z, \lambda)$ is stable for $C = I$, and that we want to be sure that this stability is maintained for some class of matrices $C \neq I$. This may be particularly relevant in the case of stochastic variants of~\eqref{eq:ode}. See \cite{BierkensKappen2012a} where this situation occurs, for perturbations $C$ within the class of positive diagonal matrices.
\end{remark}

\begin{remark}
Variations are possible in the formulation of~\eqref{eq:ode}, leading to variants of the stability condition as posed in Proposition~\ref{prop:stable_equilibrium}. However, all of these have the particular form as described above. We will not pursue this topic here any further, and hope that the contents of this section suffices to motivate the general problem formulation.
\end{remark}

\end{document}